\documentclass[12pt]{amsart}
\usepackage{graphicx}
\usepackage{amssymb}
\usepackage{epstopdf}
\usepackage{color}

\usepackage{amsmath}
\usepackage{amsthm}
\usepackage{a4wide}
\usepackage{pdfsync}
\usepackage{hyperref}

\usepackage[pagewise,mathlines]{lineno}

\newtheorem{lemma}{Lemma}[section]
\newtheorem{theorem}{Theorem}[section]
\newtheorem{remark}{Remark}

\def\rr{\mathbb{R}}

\def\eps{\epsilon}

\def\la{\langle}
\def\ra{\rangle}

\def\ov{\overline{v}}

\title[Asymptotics for nonlocal problems]{Asymptotics for nonlocal evolution problems by scaling arguments}
\author[T. I. Ignat ]{Tatiana I. Ignat }

\keywords{Nonlocal diffusion; large time behavior\\
\indent 2000 {\it Mathematics Subject Classification.} 
35B40, 45A05}

\address{Tatiana I. Ignat
\hfill\break\indent Institute of Mathematics ``Simion Stoilow'' of the Romanian Academy\\
\hfill\break\indent  21 Calea Grivitei Street \\010702 Bucharest \\ Romania }
 
 \email{{\tt tatiana.ignat@gmail.com} }

\begin{document}
\maketitle

\begin{abstract}
In this paper  we consider a nonlocal evolution problem and obtain by a scaling method the first term in the  
 asymptotic behavior of the solutions. The method employed treats in different way the smooth and the rough part of the solution.
\end{abstract}

\section{Introduction}

\label{Sect.intro}
\setcounter{equation}{0}

In this paper  we study a nonlocal equation of the form:
\begin{equation}
\label{11}
\left\{
\begin{array}{ll}
u_t(x,t) = \int_{\rr} J(x-y) (u(y,t)- u(x,t)) \, dy,& x \in \rr, \ t>0, \\[10pt]
u(x,0)=u_0 (x),& x\in \rr.
\end{array}
\right.
\end{equation}
We consider  $J:\rr\to\rr$ a nonnegative, smooth, even function rapidly decaying at inifinity, with $\int_\rr J(s)ds=1$ and the initial data $u_0\in L^1(\rr)\cap L^\infty(\rr)$.

Equations like \eqref{11} and variations of it, have been
recently widely used to model diffusion processes, for example, in
biology, dislocations dynamics, etc. For the interested reader we refer to \cite{MR1463804}, \cite{MR2120547}, \cite{MR1999098}, \cite{MR1608081} and the references therein.

In this paper we will obtain the first term in the asymptotic behavior of the solution of system \eqref{11} by using a scaling method.
The main result of this paper is the following one:
\begin{theorem} \label{teo1}
Let $u_0\in L^1 (\rr) \cap  L^\infty (\rr) $. For any $p\in [1,\infty]$
the solution $u(x,t)$ of  equation \eqref{11} satisfies:  
\begin{equation}\label{decay.u}
\lim_{t \to \infty } t^{\frac 12(1-\frac 1p)}{ \| u(t) - MG_{At}\|_{L^p(\rr)} }=0
\end{equation}
where $$G_t(x)=\frac{1}{\sqrt{4\pi t}}\exp{(-\frac{x^2}{4t})}$$
is the heat kernel and  
$$M=\int_{\rr} u_0(x)dx, \quad A=\frac{1}{2}\int_\rr J(z)z^2dz.$$
\end{theorem}

Similar results have been obtained in \cite{MR2257732} and \cite{MR2460931} by using different methods, under various assumptions on the regularity of the initial data $u_0$ and on $J$. The goal of this paper is to prove that the asymptotic behavior of the nonlocal evolution problems of type \eqref{11} can be analyzed by scaling arguments even if the equation does not support a self-similar solution due to the lack of homogeneity of the kernel $J$.

The main difficulty in applying scaling arguments in nonlocal problems is the lack of smoothness of the solution.  As observed in \cite{MR2257732}, the solution at any positive time is as smooth as the initial data is.  More precisely the 
 solution of  equation \eqref{11} can be written as 
\begin{equation}\label{decompositon}
u(x,t)=e^{-t}u_0(x)+v(x,t),
\end{equation}
 where $v$ is the smooth part of the solution while $e^{-t}u_0$ remains as smooth as the initial data is.
 By a simple computation, it follows that $v(x,t)$ verifies the equation:
\begin{equation}
\label{12}
\left\{
\begin{array}{ll}
v_t(x,t) = e^{-t}(J*u_0)(x)+(J*v-v)(x,t),& x \in \rr, \ t>0, \\[10pt]
v(x,0)=0 ,& x\in \rr.
\end{array}
\right.
\end{equation}

 The key point in using the scaling method to analyze the nonlocal model considered here
 is to apply  this method to the regular part of the solution $v$.
To obtain the decay in  Theorem \ref{teo1} we will prove a similar asymptotic behavior  for $v$:
\begin{equation}\label{limit.v}
 \lim_{t\to \infty}t^{1/2(1-1/p)}\| v(t)-MG_{At}\|_{L^p(\rr)}=0.
\end{equation}

To fix the ideas, for $v(x,t)$  solution of  problem \eqref{12} we define a family of functions $\{v_\lambda\}_{\lambda>0}$ as follows: $$v_\lambda(x,t)=\lambda v(\lambda x,\lambda^2t), \quad x\in \rr, t\geq 0.$$
In order to obtain this asymptotic behavior of $v$ we will prove that, at the time $t=1$ the rescaled family $v_\lambda(1)$ strongly converges as $\lambda\rightarrow\infty$ in the $L^p(\rr)$-norm to the solution of the heat equation, $\overline{v}_t=A\overline{v}_{xx}$ with $M\delta _0$ initial data, i.e. $MG_A$. To do that we prove that for any $0<t_1<t_2<\infty$ 
the sequence $\{v_\lambda\}$ is relatively compact on $C([t_1,t_2],L^1(\rr))$ and  that the limit point is the solution of the heat equation.

When we rescale function $v$ in fact we can write a similar scaling for $u$ with the difference that for the new family $\{u_\lambda\}$ we will not be able to prove the compactness (by the lack of regularity with respect to the initial data). Our method  not only rescale the solution but also the initial data. The limit of the rescaled solutions $u_\lambda$ when the initial data remains unchanged, i.e. the hyperbolic-parabolic relaxation limit,  has been considered in \cite[Ch.~1, p.~23]{1214.45002}.

In the context of classical diffusion problems, linear or nonlinear, the scaling method has been successfully applied. We cite here just a few references \cite{0762.35011}, \cite{MR1233647}, \cite{MR2739418}. This paper shows that the  nonlocal evolution problems involving operators as in \eqref{11}, where the smoothing effect is not present, could be treated by means of scaling methods. The extension of the method to nonlinear models as the ones analyzed in \cite{MR2356418,MR2888353} remains open.  However, the main difficulty in the context of the nonlinear problems will be to separate the smooth and rough parts of the solutions, an argument that is immediate in the case of linear problems. We recall  that there are cases when nonlinearity can help. We recall here the results in \cite{MR2138795} where a simplified model for radiating gases has been analyzed. The asymptotic profile is obtained there by using some Oleinik type estimates which are not available here.

We have considered here the case when $J$ is a smooth function rapidly decaying at infinity. In fact more general kernels can be considered. Essentially, 
as observed in \cite{MR2460931} we need the following assumptions on $J$:
\begin{equation}\label{hyp1}
\hat J(\xi)=1-A\xi^2+o(\xi^2) \quad \text{as}\, \xi\rightarrow 0
\end{equation}
 and for some $m>2$
 \begin{equation}\label{hyp2}
|\hat J(\xi)|\leq \frac {C}{|\xi|^m} \quad \text{as}\, \xi\rightarrow \infty.
\end{equation}
Obviously when $J$ is an even function and has decay faster than $1/|x|^2$ at infinity, i.e. $J\in L^1(1+|x|^2)$ for example, the first hypothesis \eqref{hyp1} is satisfied with 
$$A=\frac 12\int _{\rr} |x|^2 J(x)dx.$$
 Condition \eqref{hyp2} holds for example when $J$ has at least three derivatives in $L^1(\rr)$.
These restrictions are assumed in order to prove decay properties for the solution $v$ of system \eqref{12} and its derivative 
 in the $L^p$-norms for all $2\leq p\leq \infty$. If we only need to have estimates in the $L^2(\rr)$ norm then only $m>3/2$ is needed (see carefully the proof of Lemma 1.16 in \cite{1214.45002}).
This happens if $J$ is of class $W^{2,1}(\rr)\cap L^1(1+|x|^2)$.

 We recall here that in order to obtain $L^1-L^2$ estimates for $u$, a solution of \eqref{11}, and thus for $v$ solution of \eqref{12}, only $J\in L^1(1+|x|^2)$ is sufficient as proved by energy methods in \cite{MR2542582}, \cite{MR2103702}. The obtention of all the estimates involved in the proof by using energy methods (see \cite[Ch.~1, p.~25]{MR2656972}) for the case of the heat equation) remain to be analyzed.

\section{proof of main results}



 We first recall some preliminary results that  will help us during the proof.

We point out that as long as the initial data $u_0$ is nonnegative, $v$ a solution of system \eqref{12}  is a supersolution for system \eqref{11} with initial data identically zero. Then $v$ is nonnegative since the comparison principle holds (see \cite[Ch.~2, p.~37]{1214.45002}). We will consider here, without loss of generality, the case of nonnegative initial data $u_0$, so nonnegative solutions.

 The following lemma shows that \eqref{limit.v} is equivalent with 
the strong convergence of the sequence $\{v_\lambda(t_0)\}$ toward the heat kernel at the time $t_0$ multiplied by the mass of the initial data $MG_{t_0}.$
\begin{lemma}\label{teo2}
Let $p\in [1,\infty]$. The following statements are equivalent:
\begin{equation}\label{eq1} 
 \lim_{t\to \infty}t^{\frac 12(1-\frac 1p)}\| v(t)-MG_t\|_{L^p(\rr)}=0
 \end{equation}
and  
  \begin{equation}\label{eq2}
  v_\lambda(1)\to MG_{1} 
  \end{equation}
in the $L^p(\rr)$-norm as  $ \lambda\to \infty$, i.e. 
$$\lim_{\lambda\to\infty}\|v_\lambda(1)- MG_{1} \|_{L^p(\rr)}= 0.$$

\end{lemma}

\begin{proof}[Proof of Lemma \ref{teo2}]
Observe that at time $t_0=1$ the rescaled solution  $v_\lambda$ satisfies
\begin{align*}
\|v_\lambda(x,1) &- MG_{A}(x) \|_{L^p(\rr)}=
\|\lambda v(\lambda x,\lambda^2)-M\lambda G_{\lambda^2A}(\lambda x) \|_{L^p(\rr)}\\
&=
\lambda^{1-1/p} \|v(x,\lambda^2)-M G_{\lambda^2A}(x) \|_{L^p(\rr)}=
t^{1/2(1-1/p)}\|v(x,t)-MG_{At}(x) \|_{L^p(\rr)},
\end{align*}
where $t=\lambda^2.$ Then \eqref{eq1} holds if and only if \eqref{eq2} holds. 
\end{proof}

%
%
 
 In the following we prove 
 \begin{equation}\label{est.noua}
\lim_{\lambda\to\infty}\|v_\lambda(1)- MG_{A} \|_{L^1(\rr)}= 0.
\end{equation}
The proof is divided into four steps. We mainly follow the ideas  of \cite{MR2739418}. In  Step I we obtain estimates on $v_\lambda$ and its derivative. 
In Step II, using the Aubin-Lions compactness principle  (see for example \cite{MR916688}) we prove that $v_\lambda$ strongly converges to a function $\ov$ in $C([t_1,t_2],L^1_{loc}(\rr)).$  We then improve the convergence to $C([t_1,t_2],L^1(\rr)).$ 
In Step III we finish the proof of \eqref{est.noua} by  showing that any limit point $v$ satisfies  the heat equation with $M\delta_0$ as initial data. Since the limit point is unique then  the family $\{v_\lambda\}_{\lambda>0}$ converges to that limit. We then use \eqref{decompositon} to prove the result stated in Theorem \ref{teo1}.

 Before starting the proof of the main result let us recall that  the smooth part $v$ can be written as
$$v(x,t)=K_t\ast \varphi$$ where
\begin{equation}\label{2}
K_t(x)=\int_{\rr}(e^{t(\hat J(\xi)-1)}-e^{-t})e^{ix\xi}d\xi
\end{equation}
or in terms of the Fourier transform
\begin{equation}\label{3}
\hat{K_t}(\xi)=e^{t(\hat J(\xi)-1)}-e^{-t}.
\end{equation}
 Moreover, the rescaled solutions $\{v_\lambda\}$ satisfy the following system
 \begin{equation}
\label{v_t}
\left\{
\begin{array}{ll}
(v_\lambda)_t=\lambda^2e^{-\lambda^2t}(J_\lambda*u_{0\lambda})+\lambda^2(J_\lambda*v_\lambda-v_\lambda),& x\in\rr, t>0, \\ [10pt]
v_\lambda(x,0)=0, &x\in\rr.
\end{array}
\right.
\end{equation}
 
 Observe that the $L^1(\rr)$-norm of the nonnegative solution $v_\lambda$ is uniformly bounded by the mass of the initial data:
 \begin{equation}\label{mass}
\int _{\rr}v_\lambda (t,x)dx=(1-e^{-\lambda^2 t})\int _{\rr}u_0(x)dx.
\end{equation}

\textbf{ Step I. Estimates for $v_\lambda$.}
We  estimate the $L^p(\rr)$-norm, $p\geq 2$ of $v_\lambda$ and $(v_\lambda)_x$. Similar estimates could be obtained for $p\in[1,2)$  under stronger assumptions on function $J$ (see \cite{MR2460931}). We point out that if only the case  $p=2$ is needed then we only need to assume hypotesis \eqref{hyp2}
 with $m>1/2$ in Lemma \ref{lemma1} and $m>3/2$ in Lemma \ref{lemma2} below.

\begin{lemma}\label{lemma1}
For any  $p \in[2,\infty]$ there exists a positive constant $C(p,J)$ such that:
$$\|v_\lambda(t)\|_{L^p(\rr)}\leq C(p,J)t^{{-\frac{1}{2}}(1-\frac{1}{p})}\|u_0\|_{L^1(\rr)}$$
for any $t>0$ and any $\lambda>0$.
\end{lemma}

\begin{remark}
We emphasize that Lemma \ref{lemma1} can be proved under weaker assumptions on $J$ as in \cite{MR2542582}, \cite{MR2103702}. Essentially $J\in L^1(1+|x|^2)$ is enough to obtain bounds for $u$ solution of \eqref{11}, so for $v$ and $v_\lambda$. We point out that we do not know if the energy methods (see \cite[Ch.~1, p.~25]{MR2656972}) that work in the classical heat equation  to establish a bound of the type $\|u_x(t)\|_{L^2}\lesssim t^{-1/4}\|\varphi\|_{L^1}$ can be adapted to the nonlocal setting to obtain similar estimates for $v$ and then for $v_\lambda$. 
\end{remark}

\begin{proof}
Using the definition of $v_\lambda$ we have 
$$\|v_\lambda (t)\|_{L^p(\rr)}=\lambda ^{1-\frac 1p}\|v(\lambda ^2t)\|_{L^p(\rr)}.$$
It is then sufficient to prove the same estimate for $v$.
Using the results in \cite{MR2460931}, under the hypotheses \eqref{hyp1} and \eqref{hyp2} the kernel $K_t$ defined by \eqref{2} satisfies
$$\|K_t\|_{L^p(\rr)}\leq C(p,J)t^{-\frac{1}{2}(1-\frac{1}{p})}.$$
Therefore
$$\|v(t)\|_{L^p(\rr)}\leq C(p,J)t^{-\frac{1}{2}(1-\frac{1}{p})}\|u_0\|_{L^1(\rr)}$$
and the proof of the Lemma is finished.
\end{proof}

\begin{lemma}\label{lemma2}

For each  $p \in[2,\infty]$ there exists a positive constant C such that:
$$\|(v_\lambda(t))_x\|_{L^p(\rr)}\leq Ct^{{-\frac{1}{2}}(1-\frac{1}{p})-\frac{1}{2}}\|u_0\|_{L^1(\rr)}$$
for any $t>0$ and any $\lambda>0$.

\end{lemma}

\begin{proof}
Using the same arguments as in the previous lemma
it is  sufficient to prove that $$\|(K_t)_x\|_{L^p(\rr)}\leq Ct^{-\frac{1}{2}(1-\frac{1}{p})-\frac{1}{2}}.$$
Previous results in \cite{MR2460931} guarantee the desired estimates for $K_t$ and the proof is finished.
\end{proof}

\textbf{Step II. Compactness in $C([t_1,t_2]), L^1_{loc}(\rr)$.}
%
%
%
Let us first recall the Aubin-Lions compactness criterion (see \cite{MR916688} for related results).
\begin{theorem}\label{aubin}
Let X, B and Y be Banach spaces satisfying $X\subset B\subset Y $ with compact embedding $X\subset B.$ Assume, for $1\leq p\leq \infty$ and $T>0$, that\\ 
1) $F$ is bouned in $L^p(0, T; X)$ \\
2)$\{(f_t):f\in\ F\}$ is bounded in $L^p(0, T; Y).$

Then $F$ is relatively compact in $L^p(0, T; B)$ (and in $C(0, T; B)$ if $p=\infty$).
\end{theorem}

The following lemma gives the compactness of $\{v_\lambda\}_{\lambda>0}$ in  $C([t_1,t_2]), L^1_{loc}(\rr)$.

\begin{lemma}\label{lemma3}

For any  $0<t_1<t_2<\infty$ and for each $R>0$ the set
$$\{v_\lambda(. ,t)\}_{\lambda>0}\subseteq {C([t_1,t_2];L^1(-R,R))}$$  is relatively compact.
\end{lemma}

\begin{proof} We first prove the compactness in $C([t_1,t_2];L^2(-R,R))$ since we need estimates for $v_\lambda$ in the $L^2(\rr)$-norm and these are given by Lemma \ref{lemma2} and Lemma \ref{lemma3}. Using estimates on the $L^1$-norm of $v_\lambda$ will require more assumptions on $\hat J$ in these lemmas.

We apply the above compactness principle with $p=\infty$ and the following spaces
$X=H^1(-R,R),$  $B=L^2(-R,R)$ and $Y= H^{-1}(-R,R)$.
We prove for some $M=M(t_1,R)$ that the following estimates hold uniformly with respect to the parameter $\lambda$:
 \begin{equation}\label{est.1}
\|v_\lambda\|_{L^\infty([t_1,t_2];H^1(-R,R))}\leq M
\end{equation}
and
\begin{equation}\label{est.2}
\|(v_\lambda)_t\|_{L^\infty([t_1,t_2];H^{-1}(-R,R))}\leq M.
\end{equation}
%
%
Using Lemma \ref{lemma1} and Lemma \ref{lemma2} we immediately obtain estimate \eqref{est.1}.
%
%
%
 
We now prove the second estimate \eqref{est.2}. 
For a function $\Phi\in C^\infty_c(-R,R)$ we set  $\overline{\Phi}$  its extension as zero outside $(-R,R)$.
Using that $v_\lambda$ satisfies equation \eqref{v_t} we get:
\begin{align*}
{\la(v_\lambda)_t,\Phi}\ra_{H^{-1},H^1_0(-R,R)}&=\int^R_{-R}(v_\lambda)_t\Phi(x)dx
=\int_{\rr}{(v_\lambda)_t\overline{\Phi}(x)dx}\\
&=\int_{\rr}{[\lambda^2e^{-\lambda^2t}(J_\lambda*u_{0\lambda})(x)+\lambda^2(J_\lambda*v_\lambda-v_\lambda)(x,t)]\overline{\Phi} (x)dx}\\
&=
\lambda^2e^{-\lambda^2t}\int_{\rr}(J_\lambda*u_{0\lambda})(x)\overline{\Phi}(x)dx+\lambda^2\int_{\rr}(J_\lambda*v_\lambda-v_\lambda)(x,t)\overline{\Phi}(x)dx\\
&=B_1+B_2.
\end{align*}
Using  H\"older and Young's inequalities, we obtain the following bounds for $B_1$:
\begin{align*}
|B_1|&\leq\lambda^2e^{-\lambda^2t}\|J_\lambda*u_{0\lambda}\|_{L^2(\rr)}\|\overline{\Phi}\|_{L^2(\rr)}\leq
\lambda^2e^{-\lambda^2t}\|J_\lambda\|_{L^2(\rr)}\|u_{0\lambda}\|_{L^1(\rr)}\|\overline{\Phi}\|_{L^2(\rr)}\\
&\leq
\lambda^{3-\frac{1}{2}}e^{-\lambda^2t_1}\|J\|_{L^2(\rr)}\|u_0\|_{L^1(\rr)}\|\overline{\Phi}\|_{L^2(\rr)}\leq M\|\overline{\Phi}\|_{L^2(\rr)}=M\|\Phi\|_{L^2(-R,R)}.
\end{align*}
To obtain bounds for $B_2$, we use Cauchy's inequality, the identity 
\begin{align}\label{id1}
\int_{\rr}\int_{\rr}&{J(x-y)(\phi(y)-\phi(x))\psi(x)}dxdy\\=
\nonumber &-\frac{1}{2}\int_{\rr}\int_{\rr}{J(x-y)(\phi(y)-\phi(x))(\psi(y)-\psi(x))}dxdy.
\end{align}
and the following Lemma.
\begin{lemma}\label{lemma4}
There exists a positive constant $C(J)=\int_\rr J(z)z^2dz$ such that 
\begin{equation}\label{13}
\lambda^2\int_\rr\int_\rr{J_\lambda(x-y){(u(y,t)-u(x,t))}^2}dydx\leq C(J)\int_\rr|u_x(x)|^2dx
\end{equation}
holds for all $u\in H^1(\rr)$ and $\lambda>0.$
\end{lemma}

It follows that
\begin{align*}
B_2=
-\frac{\lambda^2}{2}\int_{\rr}\int_{\rr}J_\lambda(x-y)(v_\lambda(y,t)-v_\lambda(x,t))(\overline{\Phi}(y)-\overline{\Phi}(x))dydx.
\end{align*}
Applying Lemma \ref{lemma4} we get
\begin{align*}
|B_2|&\leq
\frac{1}{2}{\left(\lambda^2\int_{\rr}\int_{\rr}(J_\lambda(x-y){(v_\lambda(y,t)-v_\lambda(x,t))}^2dydx\right)}^\frac{1}{2}\\
&\qquad \times  {\left(\lambda^2\int_{\rr}\int_{\rr}(J_\lambda(x-y){(\overline{\Phi}(y)-\overline{\Phi}(x))}^2dydx\right)}^\frac{1}{2}\\
&\leq
C\|(v_\lambda)_x\|_{L^2(\rr)}\|\overline{\Phi}_x\|_{L^2(\rr)}.
\end{align*}
Applying Lemma \ref{lemma2} we have
\begin{align*}
B_2&\leq 
C(J)\|(v_\lambda)_x\|_{L^2(\rr)}\|{\Phi}_x\|_{L^2((-R,R))}
\leq
C(J, t_1)\|u_0\|_{L^1(\rr)}\|\Phi\|_{L^2((-R,R))}.
\end{align*}
The above estimates on $B_1$ and $B_2$ show that estimate \eqref{est.2} also holds.
By Theorem \ref{aubin} we obtain that $\{v_\lambda\}$ is relatively compact in $C([t_1,t_2];L^2(-R,R))$, then in $C([t_1,t_2];L^2(-R,R))$ 
 and the proof of Lemma \ref{lemma3} is now complete.
\end{proof}

\begin{proof}[Proof of Lemma \ref{lemma4}]
 To prove  inequality \eqref{13} we use Cauchy's inequality and Fubini's theorem.
 Let us denote by $I_\lambda$ the right hand side in \eqref{13}. 
It follows that
\begin{align*}
I_\lambda&=\lambda\int_\rr\int_\rr{J(x-y){\left(u\left(\frac{y}{\lambda}\right)-u\left(\frac{x}{\lambda}\right)\right)}^2}dydx\\ 
&=
\frac 1\lambda\int_\rr\int_\rr{J(x-y){(x-y)^2}\left[\int^1_0{u_x(\frac{x}{\lambda}+\theta\frac{y-x}{\lambda})}d\theta\right]^2}dydx\\
&\leq
\frac{1}{\lambda}\int_\rr\int_\rr{J(x-y)(y-x)^2\int^1_0{\left[u_x(\frac{x}{\lambda}+\theta\frac{y-x}{\lambda})\right]^2}d\theta}dydx\\
&=
\frac{1}{\lambda}\int_\rr{J(z)|z|^2\int^1_0\int_\rr\left[u_x(\frac{z+y}{\lambda}-\theta\frac{z}{\lambda})\right]^2dyd\theta}dz\\
&=
\int_\rr{J(z)|z|^2dz\int_\rr|u_x(x)|^2dx}
\end{align*}
and the proof of Lemma \eqref{lemma4} is finished.
\end{proof}

\textbf{Step II. Compactness in $C([t_1,t_2], L^1(\rr))$.} The previous step gives us that for any $R>0$ the family $\{v_\lambda\}$ is relatively compact in 
$C([t_1,t_2], L^1(-R,R))$. Using a standard diagonal argument the compactness in $C([t_1,t_2], L^1(\rr))$ is reduced to the fact that 
\begin{equation}\label{est.23}
\sup_{t\in[t_1,t_2 ]} \|v_\lambda (t)\|_{L^1(|x|>R)} \rightarrow 0,\quad \text{as}\ R\rightarrow \infty,
\end{equation}
uniformly on $\lambda\geq 1$. This follows from the following Lemma.

\begin{lemma}\label{est.2r}
There exists a constant $C=C(J,\|u_0\|_{L^1(\rr)})$ such that 
\begin{equation}\label{int.2r}
\int _{|x|>2R} v_\lambda (t,x)dx\leq \int _{|x|>R}(J\ast u_0)(x)dx+C(\frac{t}{R^2}+\frac {t^{1/2}}R)
\end{equation}
 holds for any $t>0$, $R>0$, uniformly for $\lambda\geq 1$.
\end{lemma}

\begin{proof}
Let  $\Psi\in C^\infty_c(\rr)$ be a nonnegative function satisfying
$$
\Psi(x)= \left\{ \begin{array}{ll}0, \qquad& |x|<1, \\
1,  \qquad & |x|>2. \end{array} \right.
$$
Put  $\Psi_R(x)=\Psi(\frac{x}{R})$ for every $R>0.$
Multiplying equation \eqref{v_t}   by $\Psi_R(x)$ and integrating in space and time we obtain:
\begin{align*}
\int_\rr v_\lambda(x,t)\Psi_R(x)dx&=\int^t_0\int_\rr\Big(\lambda^2e^{-\lambda^2s}(J_\lambda*u_{0\lambda})(x)+\lambda^2(J_\lambda*v_\lambda-v_\lambda)(x,s)\Big)\Psi_R(x)dxds\\
&=B_1+B_2.
\end{align*}
Using  identity \eqref{id1}  we obtain:
\begin{align*}
B_2&= 
-\frac{\lambda^2}{2}\int^t_0\int_\rr \int_\rr J_\lambda(x-y)[v_\lambda(y,t)-v_\lambda(x,t)][\Psi_R(y)-\Psi_R(x)]dxds\\ 
&=
\lambda^2\int^t_0\int_\rr(J_\lambda*\Psi_R-\Psi_R)(x)v_\lambda(x,s)dxds.
\end{align*}
We now need the following result.
\begin{lemma}\label{lemma5}
There exists a positive constant $C(J)$ such that 
$$\|\lambda^2(J_\lambda*\psi-\psi)\|_{L^\infty(\rr)}\leq C(J)\|\psi_{xx}\|_{L^\infty(\rr)}$$
holds for all $\lambda>0$ and $\psi\in C^2_c(\rr).$
\end{lemma}

Applying this lemma and using that the mass of $v_\lambda(t)$ is bounded by the mass of $u_0$ obtained in \eqref{mass}, we get
\begin{align*}
|B_2|
&\leq
\int^t_0\lambda^2\|J_\lambda*\Psi_R-\Psi_R\|_{L^\infty(\rr)}\|v_\lambda(s)\|_{L^1(\rr)}ds\\
&\leq C(J)t
\|u_0\|_{L^1(\rr)}\|(\Psi_R)_{xx}\|_{L^\infty(\rr)}\leq\frac{C(J)t}{R^2}\|u_0\|_{L^1(\rr)}.
\end{align*}

Now we analyze  $B_1$. Observe that
\begin{align*}
B_1&=
\lambda^2(1-e^{-\lambda^2t}) \int_\rr\int_\rr J(\lambda (x-y))u_0(\lambda y)\Psi_R(x)dydx\\
&=
(1-e^{-\lambda^2t})\int_\rr (J*u_0)(x)\Psi(\frac{x}{\lambda R})dx\\
&\leq
(1-e^{-\lambda^2t})\int_{|x|>R\lambda}|(J*u_0)(x)|dx
\leq
\int_{|x|>R}(J*u_0)(x)dx.
\end{align*}
%
The proof of Lemma \ref{est.2r} is now complete.
\end{proof}

\begin{proof}[Proof of Lemma \ref{lemma5}]
Using Taylor's formula we have for any $x$ and $y$ that
$$|\psi (x)-\psi(y)-(y-x)\psi_x(x)|\leq \frac{(y-x)^2}2\|\psi_{xx}\|_{L^\infty(\rr)}.$$
Taking into account the symmetry of $J$, we obtain:
\begin{align*}
\lambda^2\Big|\int_\rr J_\lambda(x-y)&[\psi(y)-\psi(x)]dy\Big|\\
&\leq 
\lambda^2 \Big| \int_\rr J_\lambda(x-y)(y-x)\psi '(x)dx\Big| +\frac{\lambda^2}2 \|\psi_{xx}\|_{L^\infty(\rr)}\int_\rr J_\lambda(x-y){(y-x)^2}dy\\
&=
\frac {\|\psi_{xx}\|_{L^\infty(\rr)}}2\int_\rr J(z){z^2}dz
\end{align*}
and the desired result follows.
\end{proof}

\textbf{Step III. Identification of the limit.} 
  By Step II, for any $0<t_1<t_2<\infty$, the family $\{v_\lambda\}_{\lambda>0}$ is relatively compact in $C([t_1,t_2],L^1(\rr))$. Thus, there exists a subsequence 
  $\{v_\lambda\}_{\lambda>0}$ (not relabeled) and a function $v\in C((0,\infty),L^1(\rr))$ such that 
  \begin{equation}\label{limit}
v_\lambda \rightarrow \ov \quad \text{in}\ C([t_1,t_2], L^1(\rr)) \quad \text{as}\ \lambda \rightarrow \infty.
\end{equation}
Moreover, for any $t>0$, since $\|v_\lambda(t)\|_{L^1(\rr)}\leq \|u_0\|_{L^1(\rr)}$ we also have this property for function $\ov$: $\|\ov(t)\|_{L^1(\rr)}\leq
\|u_0\|_{L^1(\rr)}$.

We multiply equation \eqref{v_t} with a function $\varphi\in C^\infty_c([0,T]\times\rr)$. Integrating  equation \eqref{3} over $[0,T]\times\rr$  we get:
\begin{align*}
\int^T_0\int_\rr{(v_\lambda)_t(x,t)\varphi (x,t)dxdt}&= \int^T_0\int_\rr \lambda^2e^{-\lambda^2t}(J_\lambda*u_{0\lambda})(x)\varphi(x,t)dxdt\\
&+
\int^T_0\int_\rr \lambda^2(J_\lambda*v_\lambda-v_\lambda)\varphi(x,t)dxdt.
\end{align*}
Integrating by parts with respect to variables $t$ and $x$ and using that $v_\lambda(x,0)=0$, $\varphi$ has compact support and  identity \eqref{id1}, we have:
\begin{align*}
-\int^T_0\int_\rr{v_\lambda(x,t)\varphi_t (x,t)dxdt}&= \int^T_0\lambda^2e^{-\lambda^2t}\int_\rr(J_\lambda*u_{0\lambda})(x)\varphi(x,t)dxdt\\
&+
\int^T_0\int_\rr \lambda^2(J_\lambda*\varphi-\varphi)v_\lambda(x,t)dxdt.
\end{align*}
We will prove later that as $\lambda \rightarrow \infty$ the following convergences hold:
\begin{equation}\label{16}
\int^T_0\int_\rr{v_\lambda(x,t)\varphi_t(x,t)}dxdt\to \int^T_0\int_\rr{\ov(x,t)\varphi_t(x,t)}dxdt,
\end{equation}
\begin{equation}\label{17}
\int^T_0\int_\rr{\lambda^2(J_\lambda*\varphi-\varphi)(x,t)v_\lambda(x,t)}dxdt\to \int^T_0\int_\rr{\ov(x,t)A\varphi_{xx}(x,t)}dxdt 
\end{equation}
and 
\begin{equation}\label{18}
\int^T_0\int_\rr{\lambda^2e^{-\lambda^2t}(J_\lambda*u_{0\lambda})(x)\varphi(x,t)}dxdt\to M\varphi(0,0),
\end{equation} 
where $A=\frac{1}{2}\int_\rr J(z)z^2dz$ and  $M=\int_\rr u_0(x)dx.$

The above results show that $v\in C((0,\infty), L^1(\rr))$ satisfies
\[
-\int _0^T\int _\rr \ov(x,t)\varphi_t(x,t)dxdt=M\varphi(0,0)+A\int _0^T \int _{\rr} \ov(x,t)\varphi_{xx}(x,t)dxdt.
\]
Hence $v$ is a solution of the heat equation
\begin{equation}
\label{heat}
\left\{
\begin{array}{ll}
\ov_t(x,t) = A\ov_{xx}(x,t) \, dy,& x \in \rr, \ t>0, \\[10pt]
\ov(0)=M\delta_0.& 
\end{array}
\right.
\end{equation}
Since this equation has a unique solution $\ov(t)=MG_{At}$, $G_t$ being the heat kernel, the whole family $\{v_\lambda\}_{\lambda>0}$ converges  to $\ov$ not only to a subsequnce. Hence 
\[
\lim _{\lambda \rightarrow \infty}\|v_\lambda (1) -MG_{A}\|_{L^1(\rr)}=0
\]
and by Lemma \ref{teo2}, $v$ the solution of system \eqref{12} satisfies
\[
\lim _{t \rightarrow \infty}\|v (t) -MG_{At}\|_{L^1(\rr)}=0.
\]
This immediately implies that $u$, the solution of system \eqref{11} satisfies 
\[
\lim _{t \rightarrow \infty}\|u(t)-MG_{At}\|_{L^1(\rr)} =0.
\]
The case $p\geq 1$ easily follows since by Step I, 
$$\|u(t)\|_{L^\infty(\rr)}\leq C(p,\|\varphi\|_{L^1(\rr)},\|\varphi\|_{L^\infty(\rr)})t^{-\frac 12}$$
and then
\[
\|u(t)-MG_{At}\|_{L^p(\rr)} \leq  \|u(t)-MG_{At}\|_{L^1(\rr)}^{1/p} (\|u(t)\|_{L^\infty(\rr)}+M\|G_{At}\|_{L^\infty(\rr)})^{1-\frac 1p}=
o(t^{-\frac 12(1-\frac 1p)}).
\]

To finish the proof of Theorem \eqref{teo1} 
 it remains to prove  \eqref{16}, \eqref{17}, \eqref{18}.
Before starting the proof of we  observe that
\begin{equation}\label{liml1}
\lim _{\lambda\rightarrow\infty}\int _0^T \|v_\lambda(t)-\ov(t)\|_{L^1(\rr)}dt=0.
\end{equation}
Indeed, for any $\eps>0$ we have
\begin{align*}
\int _0^T \|v_\lambda(t)-\ov(t)\|_{L^1(\rr)}dt&=\int _0^\eps \|v_\lambda(t)-\ov(t)\|_{L^1(\rr)}dt+\int _\eps^T \|v_\lambda(t)-\ov(t)\|_{L^1(\rr)}dt\\
&\leq 2\eps \|u_0\|_{L^1(\rr)}+ \int _\eps^T \|v_\lambda(t)-\ov(t)\|_{L^1(\rr)}dt.
\end{align*}
Since ${v_\lambda}$ is relatively compact in $C([\eps,T],L^1(\rr))$ we obtain that \eqref{liml1} holds.

Let us now prove \eqref{16}. We have
\begin{align*}
\left |\int^T_0\int_\rr{ (v_\lambda(x,t)-\ov(x,t))\varphi_t(x,t)}dxdt\right | &\leq
\int^T_0\|(v_\lambda(t)-\ov(t)\|_{L^1(\rr)}\|\varphi_t(t)\|_{L^\infty(\rr)}dt \\
&\leq
C(\varphi)\int _0^T \|v_\lambda(t)-\ov(t)\|_{L^1(\rr)}dt,
 \end{align*}
and \eqref{liml1} shows that \eqref{16} holds.

In the case of \eqref{17} we have
\begin{align*}
&\left |\int^T_0\int_\rr{\lambda^2(J_\lambda*\varphi-\varphi)(x,t)v_\lambda(x,t)}dxdt- \int^T_0\int_\rr \ov(x,t)A\varphi_{xx}(x,t)dxdt \right |\\
&\leq
\left |\int^T_0\int_\rr \lambda^2(J_\lambda*\varphi-\varphi)(x,t)(v_\lambda(x,t) -\ov(x,t))dxdt\right |\\
&\quad +
\left |\int^T_0\int_\rr \ov(x,t)  \Big( \lambda^2(J_\lambda*\varphi -\varphi)(x,t) - A\varphi_{xx}(x,t)\Big)dxdt  \right|\\
&=
A_\lambda+B_\lambda.
\end{align*}
For the first term, we have
\begin{align*}
A_\lambda&\leq
\int^T_0\int_\rr\left|\lambda^2(J_\lambda*\varphi-\varphi)(x,t)\right| \left|v_\lambda(x,t)-v(x,t)\right|dxdt\\
&\leq
\int^T_0\|\lambda^2(J_\lambda*\varphi-\varphi)(t)\|_{L^\infty(\rr)}\|v_\lambda(t)-v(t)\|_{L^1(\rr)}dt.
\end{align*}
Using Lemma \ref{lemma5} and \eqref{liml1} we obtain that $A_\lambda \rightarrow 0$ as $\lambda\rightarrow \infty$.

For the second term, $B_\lambda$, we obtain:
\begin{align*}
B_\lambda&=\int^T_0\int_\rr\left|v(x,t)\right| \left|\lambda^2(J_\lambda*\varphi-\varphi)(x,t)-A\varphi_{xx}(x,t)\right|dxdt\\
&=
\int^T_0\int_\rr\left|v(x,t)\right| \left|\lambda^3\int_\rr J(\lambda(x-y))(\varphi(y,t)-\varphi(x,t))dy-A\varphi_{xx}(x,t)\right|dxdt.
\end{align*}
Since $v$ belongs to $L^1((0,T)\times \rr)$ it is sufficient to prove that the second term in the last integral goes to zero.
For that let us observe that
\begin{align*}
\lambda^3\int_\rr &J(\lambda(x-y))(\varphi(y,t)-\varphi(x,t))dy= \lambda ^2\int _{\rr}J(z)\big(\varphi(x-\frac z\lambda) -\varphi(x)\big)dz\\
&=\lambda^2 \int _{\rr}J(z)\Big[ -\frac z\lambda \varphi_x(x) +\frac 1{\lambda^2}\int _0^1 (1-s)\varphi_{xx}(x-\frac {sz}{\lambda})z^2ds \Big]\\
&=-\frac {\varphi_x(x)}\lambda\int _{\rr}J(z)zdz+\int _{\rr}J(z)z^2 \int _0^1(1-s)\varphi_{xx}(x-\frac {sz}\lambda)ds\\
&=\int _{\rr}J(z)z^2 \int _0^1(1-s)\varphi_{xx}(x-\frac {sz}\lambda)ds\rightarrow A \varphi_{xx}(x) \quad \text{as}\, \lambda\rightarrow \infty.
\end{align*}
Using the Lebesgue dominated convergence theorem  we obtain that $B_\lambda $ goes to zero as $\lambda\rightarrow \infty$.

Before entering in the proof of \eqref{18} let us remark that
\begin{align}\label{M}
\int^\infty_0e^{-t}\int_\rr(J*u_0)(x)dxdt&=\int _{\rr}u_0(x)dx=M.
\end{align}
Since $\varphi$ has compact support we have that
\begin{equation}\label{20}
\left|\varphi\left(\frac{x}{\lambda},\frac{t}{\lambda^2}\right)-\varphi(0,0)\right|\leq \left(\frac{|x|}{\lambda}+\frac{t}{\lambda^2}\right)\|\nabla\varphi\|_{L^\infty(\rr^2)}
\leq
\frac{C(\varphi)}{\lambda}
\end{equation}
Using \eqref{M} and \eqref{20} we have that when $\lambda\rightarrow \infty$
\begin{align*}
\Big|\int^\infty_0\int_\rr   \lambda^2 & e^{-\lambda^2t}(J_\lambda*u_{0\lambda})(x)\varphi(x,t)dxdt-M\varphi(0,0) \Big|\\
&=
\left|\int^\infty_0\int_\rr{\lambda^2e^{-\lambda^2t}\lambda(J*u_0)(\lambda x)\varphi(x,t)}dxdt-M\varphi(0,0)\right|\\
&\leq
\int^\infty_0e^{-t}\int_\rr(J*u_0)(x)\left|\varphi(\frac{x}{\lambda},\frac{t}{\lambda^2})-\varphi(0,0) \right|dxdt\\
&\leq
\frac{C(\varphi)}{\lambda}\int^\infty_0e^{-t}\int_\rr(J*u_0)(x)dxdt=
\frac{C(\varphi)}{\lambda} M\to 0.
\end{align*}
The proof  of \eqref{16}, \eqref{17}, \eqref{18}  is now  complete.

\medskip

 {\bf
 Acknowledgements.}

The  author was  supported by Grant PN-II-ID-PCE-2012-4-0021 "Variable Exponent Analysis: Partial Differential Equations and Calculus of Variations" of the Romanian National Authority for Scientific Research, CNCS -- UEFISCDI and by a doctoral fellowship of IMAR. Parts of this paper have been done during the visit of the author at BCAM-Basque Center for Applied Mathematics, Bilbao, Spain. The author thanks the center for hospitality and support.  

%
%

%
%
\bibliographystyle{plain}
\bibliography{biblio}

\begin{thebibliography}{10}

\bibitem{1214.45002}
Fuensanta Andreu-Vaillo, Jos\'e~M. Maz\'on, Julio~D. Rossi, and J.~Juli\'an
  Toledo-Melero.
\newblock {\em {Nonlocal diffusion problems.}}
\newblock {Mathematical Surveys and Monographs 165. Providence, RI: American
  Mathematical Society (AMS); Madrid: Real Sociedad Matem\'atica Espa\~nola.
  xv, 256~p. \$~82.00 }, 2010.

\bibitem{MR1463804}
Peter~W. Bates, Paul~C. Fife, Xiaofeng Ren, and Xuefeng Wang.
\newblock Traveling waves in a convolution model for phase transitions.
\newblock {\em Arch. Rational Mech. Anal.}, 138(2):105--136, 1997.

\bibitem{MR2120547}
C.~Carrillo and P.~Fife.
\newblock Spatial effects in discrete generation population models.
\newblock {\em J. Math. Biol.}, 50(2):161--188, 2005.

\bibitem{MR2257732}
E.~Chasseigne, M.~Chaves, and J.~D. Rossi.
\newblock Asymptotic behavior for nonlocal diffusion equations.
\newblock {\em J. Math. Pures Appl. (9)}, 86(3):271--291, 2006.

\bibitem{MR2888353}
Qiang Du, James~R. Kamm, R.~B. Lehoucq, and Michael~L. Parks.
\newblock A new approach for a nonlocal, nonlinear conservation law.
\newblock {\em SIAM J. Appl. Math.}, 72(1):464--487, 2012.

\bibitem{0762.35011}
M.~Escobedo and E.~Zuazua.
\newblock {Large time behavior for convection-diffusion equations in
  $\mathbf{R}^N$.}
\newblock {\em J. Funct. Anal.}, 100(1):119--161, 1991.

\bibitem{MR1233647}
Miguel Escobedo, Juan~Luis V{\'a}zquez, and Enrike Zuazua.
\newblock Asymptotic behaviour and source-type solutions for a
  diffusion-convection equation.
\newblock {\em Arch. Rational Mech. Anal.}, 124(1):43--65, 1993.

\bibitem{MR1999098}
P.~Fife.
\newblock Some nonclassical trends in parabolic and parabolic-like evolutions.
\newblock In {\em Trends in nonlinear analysis}, pages 153--191. Springer,
  Berlin, 2003.

\bibitem{MR1608081}
Paul~C. Fife and Xuefeng Wang.
\newblock A convolution model for interfacial motion: the generation and
  propagation of internal layers in higher space dimensions.
\newblock {\em Adv. Differential Equations}, 3(1):85--110, 1998.

\bibitem{MR2656972}
Mi-Ho Giga, Yoshikazu Giga, and J{\"u}rgen Saal.
\newblock {\em Nonlinear partial differential equations}.
\newblock Progress in Nonlinear Differential Equations and their Applications,
  79. Birkh\"auser Boston Inc., Boston, MA, 2010.
\newblock Asymptotic behavior of solutions and self-similar solutions.

\bibitem{MR2356418}
Liviu~I. Ignat and Julio~D. Rossi.
\newblock A nonlocal convection-diffusion equation.
\newblock {\em J. Funct. Anal.}, 251(2):399--437, 2007.

\bibitem{MR2460931}
Liviu~I. Ignat and Julio~D. Rossi.
\newblock Refined asymptotic expansions for nonlocal diffusion equations.
\newblock {\em J. Evol. Equ.}, 8(4):617--629, 2008.

\bibitem{MR2542582}
Liviu~I. Ignat and Julio~D. Rossi.
\newblock Decay estimates for nonlocal problems via energy methods.
\newblock {\em J. Math. Pures Appl. (9)}, 92(2):163--187, 2009.

\bibitem{MR2739418}
Grzegorz Karch and Kanako Suzuki.
\newblock Spikes and diffusion waves in a one-dimensional model of chemotaxis.
\newblock {\em Nonlinearity}, 23(12):3119--3137, 2010.

\bibitem{MR2138795}
Philippe Lauren{\c{c}}ot.
\newblock Asymptotic self-similarity for a simplified model for radiating
  gases.
\newblock {\em Asymptot. Anal.}, 42(3-4):251--262, 2005.

\bibitem{MR2103702}
Denis Serre.
\newblock {$L^1$}-stability of nonlinear waves in scalar conservation laws.
\newblock In {\em Evolutionary equations. {V}ol. {I}}, Handb. Differ. Equ.,
  pages 473--553. North-Holland, Amsterdam, 2004.

\bibitem{MR916688}
Jacques Simon.
\newblock Compact sets in the space {$L^p(0,T;B)$}.
\newblock {\em Ann. Mat. Pura Appl. (4)}, 146:65--96, 1987.

\end{thebibliography}

\end{document}